\documentclass{amsart}
\usepackage{latexsym,amssymb,amsthm,amsmath}

\usepackage{amsthm}
\usepackage{amssymb}
\usepackage{amsmath}

\theoremstyle{plain}
\newtheorem{theorem}{Theorem}[section]
\newtheorem*{Theorem B}{Theorem B}
\newtheorem*{Theorem A}{Theorem A}
\newtheorem{lemma}{Lemma}[section]

\newtheorem{corollary}{Corollary}[section]

\newtheorem{definition}{Definition}[section]

\numberwithin{equation}{section}

\newtheorem{remark}{Remark}[section]
\begin{document}
\title[A note on submanifolds of $\bar{M}^{2n+1}(f_1,f_2,f_3)$ with respect $\cdots$]{A note on submanifolds of $\bar{M}^{2n+1}(f_1,f_2,f_3)$ with respect to certain connections}
\author[P. Mandal and S. K. Hui]{Pradip Mandal and Shyamal Kumar Hui$^{*}$}
\keywords{generalized Sasakian-space-forms, semisymmetric metric connection, semisymmetric non-metric connection, Schouten-Van Kampen Connection, Tanaka-Webster connection.\\{* corresponding author}}

\begin{abstract}
The present paper deals with some results of almsot semi-invariant submanifolds of generalized Sasakian-space-forms in \cite{ALEGRE3}
with respect to semisymmetric metric connection, semisymmetric non-metric connection,
Schouten-van Kampen connection and Tanaka-Webster connection.
\end{abstract}
\maketitle
\section{Introduction}
  The notion of generalized
Sasakian-space-form  was introduced by Alegre et al. \cite{ALEGRE1}. An almost contact metric manifold
$\bar{M}(\phi,\xi,\eta,g)$
whose curvature tensor $\bar{R}$ satisfies
\begin{align}
\label{eqn1.1}
\bar{R}(X,Y)Z &=f_1\big\{g(Y,Z)X-g(X,Z)Y\big\}+f_2\big\{g(X,\phi Z)\phi Y\\
\nonumber& - g(Y,\phi Z)\phi X + 2g(X,\phi Y)\phi Z\big\}+f_3\big\{\eta(X)\eta(Z)Y\\
\nonumber& - \eta(Y)\eta(Z)X+g(X,Z)\eta(Y)\xi - g(Y,Z)\eta(X)\xi\big\}
\end{align}
for all vector fields $X$, $Y$, $Z$ on $\bar{M}$ and $f_1,f_2,f_3$ are certain smooth functions on
 $\bar{M}$ is said to be generalized Sasakian-space-form \cite{ALEGRE1}. Such a manifold of dimension
$(2n+1)$, $n>1$ (the condition $n>1$ is assumed throughout the paper),
is denoted by $\bar{M}^{2n+1}(f_1,f_2,f_3)$ \cite{ALEGRE1}. Many authors studied this space form with different aspects.
For this, we may refer, (\cite{PM1}, \cite{HUI1}, \cite{HUI2}, \cite{HUI3}, \cite{HUI4}, \cite{HUI5}, \cite{HUI6}, \cite{KISH} and \cite{HUI8}).
It reduces to Sasakian-space-form if $f_1 = \frac{c+3}{4}$, $f_2 = f_3 = \frac{c-1}{4}$ \cite{ALEGRE1}. We denote Sasakian-space-form of dimension $(2n+1)$ by $M^{2n+1}(c)$.

After introduced the semisymmetric linear connection by Friedman and Schouten \cite{FRID},
Hayden \cite{HAYD} gave the idea of metric connection with torsion on a Riemannian manifold.
Later, Yano \cite{YANO} and many others (see, \cite{HUI7}, \cite{SHAIKH1}, \cite{SULAR} and references therein)
studied semisymmetric metric connection in different context. The idea of semisymmetric non-metric
connection was introduced by Agashe and Chafle \cite{AGAS}.

The Schouten-van Kampen connection introduced
 for the study of non-holomorphic manifolds (\cite{SCHO}, \cite{VRAN}). In $2006$, Bejancu \cite{BEJA3}
studied Schouten-van Kampen connection on foliated manifolds. Recently Olszak \cite{OLSZ}
studied Schouten-van Kampen connection on almost(para) contact metric structure.

The Tanaka-Webster connection (\cite{TANA}, \cite{WEBS}) is the canonical affine connection defined on a non-degenerate
pseudo-Hermitian CR-manifold. Tanno \cite{TANN} defined the Tanaka-Webster connection for contact metric manifolds.

In \cite{ALEGRE3}, Alegre and Carriazo studied almost semi-invariant submanifolds of generalized Sasakian-space-form with respect to Levi-Civita connection. In this paper, we have studied the results of \cite{ALEGRE3} with respect to certain connections, namely semisymmetric metric connection, semisymmetric non-metric connection, Schouten-van Kampen connection, Tanaka-Webster connection.
\section{preliminaries}
In an almost contact metric manifold $\bar{M}(\phi,\xi,\eta,g)$, we have \cite{BLAIR}
\begin{align}
\label{eqn2.1}
\phi^2(X) = -X+\eta(X)\xi,\ \phi \xi=0,
\end{align}
\begin{align}
\label{eqn2.2}
\eta(\xi) = 1,\ g(X,\xi) = \eta(X),\ \eta(\phi X) = 0,
\end{align}
\begin{align}
\label{eqn2.3}
g(\phi X,\phi Y) = g(X,Y)-\eta(X)\eta(Y),
\end{align}
\begin{align}
\label{eqn2.4} g(\phi X,Y) = -g(X,\phi Y).
\end{align}
In
$\bar{M}^{2n+1}(f_1,f_2,f_3)$, we have \cite{ALEGRE1}
\begin{align}
\label{eqn2.5}
(\bar{\nabla}_X\phi)(Y) = (f_1-f_3)[g(X,Y)\xi - \eta(Y)X],
\end{align}
\begin{align}
\label{eqn2.6}
\bar{\nabla}_X\xi = -(f_1-f_3) \phi X,
\end{align}
where $\bar{\nabla}$ is the Levi-Civita connection of $\bar{M}^{2n+1}(f_1,f_2,f_3)$.

The semisymmetric metric connection $\tilde{\bar{\nabla}}$ and the Riemannian connection $\bar{\nabla}$
on ${\bar{M}}^{2n+1}(f_{1},f_{2},f_{3})$ are related by \cite{YANO}
\begin{align}
\label{eqn2.10}
 \tilde{\bar{\nabla}}_{X}Y= \bar{\nabla}_X Y+\eta(Y)X-g(X,Y)\xi.
\end{align}
The Riemannian curvature tensor $\tilde{\bar{R}}$ of  $\bar{M}^{2n+1}(f_{1},f_{2},f_{3})$ with respect to $\tilde{\bar{\nabla}}$ is
\begin{eqnarray}
    \label{eqn2.11}
\tilde{\bar{R}}(X,Y)Z &=&(f_1-1)\big\{g(Y,Z)X-g(X,Z)Y\big\}+f_2\big\{g(X,\phi Z)\phi Y\\
\nonumber&-&g(Y,\phi Z)\phi X + 2g(X,\phi Y)\phi Z\big\}+(f_3-1)\big\{\eta(X)\eta(Z)Y \\
\nonumber&-& \eta(Y)\eta(Z)X+g(X,Z)\eta(Y)\xi - g(Y,Z)\eta(X)\xi\big\}\\
\nonumber&+&(f_1-f_3)\{g( X, \phi Z)Y-g( Y,\phi Z)X\\
\nonumber&+&g(Y,Z)\phi X-g(X,Z)\phi Y\}.
\end{eqnarray}
The semisymmetric non-metric connection ${\bar{\nabla}}^{'}$ and the Riemannian connection $\bar{\nabla}$
on ${\bar{M}}^{2n+1}(f_{1},f_{2},f_{3})$ are related by \cite{AGAS}
\begin{align}
\label{eqn2.12}
\bar{\nabla}^{'}_{X}Y= \bar{\nabla}_X Y+\eta(Y)X.
\end{align}
The Riemannian curvature tensor ${\bar{R}}^{'}$ of  $\bar{M}^{2n+1}(f_{1},f_{2},f_{3})$ with respect to ${\bar{\nabla}}^{'}$ is
\begin{eqnarray}
\label{eqn2.13}
{\bar{R}}^{'}(X,Y)Z &=&f_1\big\{g(Y,Z)X-g(X,Z)Y\big\}+f_2\big\{g(X,\phi Z)\phi Y\\
\nonumber&-&g(Y,\phi Z)\phi X + 2g(X,\phi Y)\phi Z\big\}+f_3\big\{\eta(X)\eta(Z)Y\\
\nonumber& -&\eta(Y)\eta(Z)X+g(X,Z)\eta(Y)\xi - g(Y,Z)\eta(X)\xi\big\}\\
\nonumber&+&(f_1-f_3)[g(X,\phi  Z)Y-g( Y,\phi Z) X]\\
\nonumber&+&\eta(Y)\eta(Z)X-\eta(X)\eta(Z)Y.
\end{eqnarray}
The Schouten-van Kampen connection $\hat{\bar{\nabla}}$ and the Riemannian connection $\bar{\nabla}$
of ${\bar{M}}^{2n+1}(f_{1},f_{2},f_{3})$ are related by \cite{OLSZ}
\begin{align}
\label{eqn2.14}
\hat{\bar{\nabla}}_{X}Y=\bar{\nabla}_X Y+(f_1-f_3)\eta(Y)\phi X-(f_1-f_3)g(\phi X,Y)\xi.
\end{align}
The Riemannian curvature tensor $\hat{\bar{R}}$ of
$\bar{M}^{2n+1}(f_{1},f_{2},f_{3})$ with respect to $\hat{\bar{\nabla}}$ is
\begin{eqnarray}
\label{eqn2.15}
\hat{\bar{R}}(X,Y)Z&=&f_1\big\{g(Y,Z)X-g(X,Z)Y\big\}+f_2\big\{g(X,\phi Z)\phi Y\\
\nonumber&-&g(Y,\phi Z)\phi X +2g(X,\phi Y)\phi Z\big\}\\
\nonumber&+&\{f_3+(f_1-f_3)^2\}\big\{\eta(X)\eta(Z)Y - \eta(Y)\eta(Z)X\\
\nonumber&+&g(X,Z)\eta(Y)\xi - g(Y,Z)\eta(X)\xi\big\}\\
\nonumber&+&(f_1-f_3)^2\{g(X,\phi Z)\phi Y-g(Y, \phi Z)\phi X\}.
\end{eqnarray}
The Tanaka-Webster connection $\stackrel{\ast}{\bar{\nabla}}$ and the Riemannian connection $\bar{\nabla}$
of ${\bar{M}}^{2n+1}(f_{1},f_{2},f_{3})$ are related by \cite{CHO}
\begin{align}
\label{eqn2.16}
 \stackrel{\ast}{\bar{\nabla}}_{X}Y= \bar{\nabla}_X Y+\eta(X)\phi Y+(f_1-f_3)\eta(Y)\phi X-(f_1-f_3)g(\phi X,Y)\xi.
\end{align}
The Riemannian curvature tensor $\stackrel{*}{\bar{R}}$ of
$\bar{M}^{2n+1}(f_{1},f_{2},f_{3})$ with respect to $\stackrel{*}{\bar{\nabla}}$ is
\begin{eqnarray}
\label{eqn2.17}
\stackrel{*}{\bar{R}}(X,Y)Z &=&f_1\big\{g(Y,Z)X-g(X,Z)Y\big\}+f_2\big\{g(X,\phi Z)\phi Y\\
\nonumber&-& g(Y,\phi Z)\phi X+ 2g(X,\phi Y)\phi Z\big\}\\
\nonumber&+&\{f_3+(f_1-f_3)^2\}\big\{\eta(X)\eta(Z)Y - \eta(Y)\eta(Z)X\\
\nonumber&+&g(X,Z)\eta(Y)\xi - g(Y,Z)\eta(X)\xi\big\}\\
\nonumber&+&{(f_1-f_3)^2}\{g(X,\phi Z)\phi Y-g(Y,\phi Z)\phi X\}\\
\nonumber&+&2(f_1-f_3)g(X,\phi Y)\phi Z.
\end{eqnarray}

\indent Let $M$ be a $(m+1)$-dimensional submanifold of
$\bar{M}^{2n+1}(f_1,f_2,f_3)$. If $\nabla$ and $\nabla^\perp$ are the induced
connections on the tangent bundle $TM$ and the normal bundle $T^\perp{M}$ of $M$, respectively
then the Gauss and Weingarten formulae are given by \cite{YANOKON1}
\begin{align}
\label{eqn2.7}
\bar{\nabla}_XY = \nabla_XY +h(X,Y),\ \bar{\nabla}_XV = -A_VX + \nabla_X^{\perp}V
\end{align}
for all $X,Y\in\Gamma(TM)$ and $V\in\Gamma(T^{\perp}M)$, where $h$ and $A_V$ are second fundamental form and shape operator (corresponding to the normal vector field V), respectively and they are related by $ g(h(X,Y),V) = g(A_VX,Y)$ \cite{YANOKON1}.

 Moreover, if $h(X,Y)=0$ for all $X,Y \in \Gamma(TM)$
 then $M$ is said to be {\it{totally geodesic}} and if $H=0$ then $M$ is {\it{minimal}} in $\bar{M}$, where $H$ is the mean curvature tensor.

 From (\ref{eqn2.7}), we have the Gauss equations as
\begin{align}
\label{eqn2.9}
\bar{R}(X,Y,Z,W)&=R(X,Y,Z,W)-g\big(h(X,W),h(Y,Z)\big)\\
\nonumber&+g\big(h(X,Z),h(Y,W)\big),
\end{align}
where $R$ is the curvature tensor of $M$.
Let $\tilde{\nabla}$, $\nabla^{'}$, $\hat{\nabla}$ and  $\stackrel{*}{\nabla}$ are the induced connection of $M$ from the connection $\tilde{\bar{\nabla}}$, $\bar{\nabla}^{'}$, $\hat{\bar{\nabla}}$ and  $\stackrel{*}{\bar{\nabla}}$ of $\bar{M}^{2n+1}(f_1,f_2,f_3)$ respectively.
Then Gauss equation with respect to $\tilde{\bar{\nabla}}$, $\bar{\nabla}^{'}$, $\hat{\bar{\nabla}}$ and  $\stackrel{*}{\bar{\nabla}}$ are
\begin{align}
\label{eqn2.9a}
\tilde{\bar{R}}(X,Y,Z,W)&=\tilde{R}(X,Y,Z,W)-g\big(\tilde{h}(X,W),\tilde{h}(Y,Z)\big)\\
\nonumber&+g\big(\tilde{h}(X,Z),\tilde{h}(Y,W)\big),
\end{align}
\begin{align}
\label{eqn2.9b}
\bar{R}^{'}(X,Y,Z,W)&=R^{'}(X,Y,Z,W)-g\big(h^{'}(X,W),h^{'}(Y,Z)\big)\\
\nonumber&+g\big(h^{'}(X,Z),h^{'}(Y,W)\big),
\end{align}
\begin{align}
\label{eqn2.9c}
\hat{\bar{R}}(X,Y,Z,W)&=\hat{R}(X,Y,Z,W)-g\big(\hat{h}(X,W),\hat{h}(Y,Z)\big)\\
\nonumber&+g\big(\hat{h}(X,Z),\hat{h}(Y,W)\big),
\end{align}
\begin{align}
\label{eqn2.9d}
\stackrel{*}{\bar{R}}(X,Y,Z,W)&=\stackrel{*}{R}(X,Y,Z,W)-g\big(\stackrel{*}{h}(X,W),\stackrel{*}{h}(Y,Z)\big)\\
\nonumber&+g\big(\stackrel{*}{h}(X,Z),\stackrel{*}{h}(Y,W)\big),
\end{align}
where $\tilde{h}$, $h^{'}$, $\hat{h}$, $\stackrel{*}{h}$ are the second fundamental forms with respect to $\tilde{\nabla}$, ${\nabla}^{'}$, $\hat{\nabla}$ and  $\stackrel{*}{\nabla}$ respectively. Also $\tilde{H}$, $H^{'}$, $\hat{H}$, $\stackrel{*}{H}$ be the mean curvature of $M$ with respect to $\tilde{\nabla}$, ${\nabla}^{'}$, $\hat{\nabla}$ and  $\stackrel{*}{\nabla}$ respectively.

For any $X\in\Gamma(TM)$, we may write
\begin{equation}
\label{eqn2.10a} \phi X=TX+FX,
\end{equation}
where $TX$ is the tangential component and $FX$ is the
normal component of $\phi X$.

\begin{definition}(\cite{ALEGRE3}, \cite{TRIPs}) A submanifold $M$ of an almost contact metric manifold $\bar{M}$, $\xi$ tangent to $M$, is said to be an almost semi-invariant submanifold if their exist $l$ functions $\lambda_1,\cdots,\lambda_l$, defined on $M$ with values in $(0,1)$, such that \begin{enumerate}
  \item [(i)] $-\lambda_1^2(p),\cdots,-\lambda_l^2(p)$ are distinct eigenvalues of $T^2|_D$ at $p\in M$, with
  \begin{equation*}
    T_pM=D^1_p\oplus D^0_p\oplus D^{\lambda_1}_p\oplus\cdots\oplus D^{\lambda_l}_p\oplus span\{\xi_p\},
  \end{equation*}
  where $D^\lambda_p$, $\lambda\in \{1,0,\lambda_1(p),\cdots,\lambda_l(p)\}$, denotes the eigenspace associated to the eigenvalue $-\lambda^2$.
  \item [(ii)] the dimension of $D^1_p,D^0_p, D^{\lambda_1}_p,\cdots, D^{\lambda_l}_p$ are independent of $p\in M$.
\end{enumerate}
Let the orthogonal projection from $TM$ on $D^\lambda$ be $U^\lambda$. Then we have
\begin{equation}\label{eqn2.10b}
g(TX,TY)=\sum_{\lambda}^{}\lambda^2g(U^\lambda X,U^\lambda Y).
\end{equation}
\end{definition}

Let us consider $\{E_1,\cdots,E_m,E_{m+1}=\xi\}$ and $\{F_1,\cdots,F_{2n-m}\}$ local orthonormal basis of $TM$ and $T^\bot M$ respectively, and denote $A_{F_k}=A_k$.
\section{Ricci tensor on $ M$ of $\bar{M}^{2n+1}(f_1,f_2,f_3)$ with $\tilde{\bar{\nabla}}$}
\begin{lemma}
The Ricci tensor $\tilde{S}$ of submanifold $M$ of $\bar{M}^{2n+1}(f_1,f_2,f_3)$ with respect to $\tilde{\bar{\nabla}}$ is
\begin{eqnarray}
\label{eqn3.1}
 \tilde{S}(X,Y) &=& mf_1g(X,Y)+3f_2g(TX,TY)-(f_3-1)\{g(X,Y)\\
  \nonumber&+&(m-1)\eta(X)\eta(Y)\}+(f_1-f_3)(m-1)g(TX,Y) \\
\nonumber&+&\sum_{k=1}^{2n-m}\{(m+1)(trace \tilde{A}_k)g(\tilde{A}_kX,Y)-g(\tilde{A}_kX,\tilde{A}_kY)\}
\end{eqnarray}
for any vector fields $X,Y$ on $M$.
\begin{proof}
Using (\ref{eqn2.11}) and (\ref{eqn2.9a}) we have the above Lemma.
\end{proof}
\end{lemma}
\begin{lemma}
The Ricci tensor $\tilde{S}$ of almost semi-invariant submanifolds $M$ of $\bar{M}^{2n+1}(f_1,f_2,f_3)$ with respect to $\tilde{\bar{\nabla}}$ is
\begin{eqnarray}
\label{eqn3.2}
\tilde{S}(X,Y)&=&\sum_{\lambda}^{}(mf_1+3f_2\lambda^2-f_3+1)g(U^\lambda X,U^\lambda Y)\\
\nonumber&&+m(f_1-f_3+1)\eta(X)\eta(Y)+(f_1-f_3)(m-1)g(T X,Y)\\
\nonumber&+&\sum_{k=1}^{2n-m}\{(m+1)(trace \tilde{A}_k)g(\tilde{A}_kX,Y)-g(\tilde{A}_kX,\tilde{A}_kY)\}
\end{eqnarray}
for any vector fields $X,Y$ on $M$.
\end{lemma}
\begin{proof}
Using (\ref{eqn2.10b}) and (\ref{eqn3.1}) we have the above Lemma.
\end{proof}
\begin{corollary}
For an almost semi-invariant submanifolds $M$ of Sasakian-space-form $\bar{M}^{2n+1}(c)$ with respect to $\tilde{\bar{\nabla}}$ is
\begin{eqnarray}
\label{eqn3.3}
\tilde{S}(X,Y) &=& \frac{(m-1+3\lambda^2)c+3(m-\lambda^2)+5}{4}g(U^\lambda X,U^\lambda Y) \\
 \nonumber&+&(m-1)g(T X,Y)+\sum_{k=1}^{2n-m}\{(m+1)(trace \tilde{A}_k)g(\tilde{A}_kX,Y)\\
 \nonumber&-&g(\tilde{A}_kX,\tilde{A}_kY)\}
\end{eqnarray}
for any vector fields $X,Y$ on $M$.
\end{corollary}
\begin{proof}
Putting $f_1=\frac{c+3}{4},\  f_2=f_3=\frac{c-1}{4}$ in (\ref{eqn3.2}) we obtain the result.
\end{proof}
\begin{lemma}
The scalar curvature $\tilde{\tau}$ of an almost semi-invariant submanifold $M$ of $\bar{M}^{2n+1}(f_1,f_2,f_3)$ with respect to $\tilde{\bar{\nabla}}$ is
\begin{eqnarray}
\label{eqn3.4}
\tilde{\tau} &=& f_1+\frac{1}{m(m+1)}\{3f_2\sum_{\lambda}^{}n(\lambda)\lambda^2-2mf_3+2m\} \\
 \nonumber&+&(m+1)^2||\tilde{H}||^2-||\tilde{h}||^2.
\end{eqnarray}
\end{lemma}
\begin{proof}
Let us consider an orthonormal frame $\{E_1,\cdots,E_{n(\lambda)}\}$ in $D^\lambda$. Then we have
\begin{equation}
\label{eqn3.7}
\tilde{\tau}=\frac{1}{m(m+1)}\sum\limits_{i,j=1}^{m+1}\tilde{R}(E_i,E_j,E_j,E_i).
\end{equation}
Using (\ref{eqn2.11}), (\ref{eqn2.9a}) in (\ref{eqn3.7}) we get (\ref{eqn3.4}).
\end{proof}
\begin{theorem}
If $M$ is an almost semi-invariant minimal submanifold of $\bar{M}^{2n+1}(f_1,f_2,f_3)$ with respect to $\tilde{\bar{\nabla}}$, then the following relation holds:
\begin{itemize}
\item [(i)]
$\tilde{S}(X,X)\leq\sum\limits_{\lambda}^{}(mf_1+3f_2\lambda^2-f_3+1)g(U^\lambda X,U^\lambda X)$
$+m(f_1-f_3+1)\eta(X)\eta(X)+(f_1-f_3)(m-1)g(TX,X)$,\\

\item[(ii)] $\tilde{\tau}\leq f_1+\frac{1}{m(m+1)}\{3f_2\sum\limits_{\lambda}^{}n(\lambda)\lambda^2-2m(f_3-1)\}$.
\end{itemize}
\end{theorem}
\begin{proof}
Since $M$ is minimal submanifold with respect to $\tilde{\bar{\nabla}}$, then we have
\begin{eqnarray}\label{eqn3.5}
\sum_{k=1}^{2n-m}(trace\ \tilde{A}_k)g(\tilde{A}_kX,X)&=&\sum_{i=1}^{m+1}g(\tilde{h}(X,X),\tilde{h}(E_i,E_i))\\
\nonumber&=&(m+1)g(\tilde{h}(X,X),\tilde{H})=0.
\end{eqnarray}
Using (\ref{eqn3.5}) in (\ref{eqn3.2}) we have
\begin{eqnarray}
\label{eqn3.6}
  &&\tilde{S}(X,X)-\sum_{\lambda}^{}(mf_1+3f_2\lambda^2-f_3+1)g(U^\lambda X,U^\lambda X)\\
\nonumber&&-m(f_1-f_3+1)\eta(X)\eta(X)-(f_1-f_3)(m-1)g(T X,X)\\
\nonumber&=&-\sum_{k=1}^{2n-m}g(\tilde{A}_kX,\tilde{A}_kX)\leq 0,
\end{eqnarray}
which proves (i).\\
The second part (ii) comes directly from Lemma $3.3$.
\end{proof}
\begin{remark}
  The equality of (i) and (ii) in Theorem $3.1$ holds if $M$ is almost semi-invariant totally geodesic submanifolds of $\bar{M}^{2n+1}(f_1,f_2,f_3)$ with respect to $\tilde{\bar{\nabla}}$.
\end{remark}
\begin{proof}
If $M$ is totally geodesic submanifold with respect to $\tilde{\bar{\nabla}}$, then $M$ is minimal submanifold with respect to $\tilde{\bar{\nabla}}$.
Then by virtue of Lemma $3.2$ we have the equality case (i) and by virtue of Lemma $3.3$ we have equality case of (ii).
\end{proof}
\section{Submanifolds of $\bar{M}^{2n+1}(f_1,f_2,f_3)$ with ${\bar{\nabla}}^{'}$}
\begin{lemma}
The Ricci tensor ${S}^{'}$ of submanifold $M$ of $\bar{M}^{2n+1}(f_1,f_2,f_3)$ with respect to ${\bar{\nabla}}^{'}$ is
\begin{eqnarray}
\label{eqn4.1}
{S}^{'}(X,Y) &=& mf_1g(X,Y)+3f_2g(TX,TY)-(f_3-1)\{g(X,Y)\\
\nonumber&+&(m-1)\eta(X)\eta(Y)\}-m\eta(X)\eta(Y)+(f_1-f_3)g(TX,Y) \\
\nonumber&+&\sum_{k=1}^{2n-m}\{(m+1)(trace A^{'}_k)g(A^{'}_kX,Y)-g(A^{'}_kX,A^{'}_kY)\}
\end{eqnarray}
for any vector fields $X,Y$ on $M$.
\end{lemma}
\begin{proof}
Using (\ref{eqn2.13}) and (\ref{eqn2.9b})  we have the above Lemma.
\end{proof}
\begin{lemma}
The Ricci tensor $S^{'}$ of almost semi-invariant submanifolds $M$ of $\bar{M}^{2n+1}(f_1,f_2,f_3)$ with respect to $\bar{\nabla}^{'}$ is
\begin{eqnarray}
\label{eqn4.2}
S^{'}(X,Y)&=&\sum_{\lambda}^{}(mf_1+3f_2\lambda^2-f_3+1)g(U^\lambda X,U^\lambda Y)\\
\nonumber&+&m(f_1-f_3)\eta(X)\eta(Y)+(f_1-f_3)g(T X,Y)\\
\nonumber&+&\sum_{k=1}^{2n-m}\{(m+1)(traceA^{'}_k)g(A^{'}_kX,Y)-g(A^{'}_kX,A^{'}_kY)\}
\end{eqnarray}
for any vector fields $X,Y$ on $M$.
\end{lemma}
\begin{proof}
Using (\ref{eqn2.10b}) and (\ref{eqn4.1}) we have the above Lemma.
\end{proof}
\begin{corollary}
The Ricci tensor $S^{'}$ of almost semi-invariant submanifolds $M$ of $\bar{M}^{2n+1}(c)$ with respect to $\bar{\nabla}^{'}$ is
\begin{eqnarray}
\label{eqn4.3}
\ \ \ {S}^{'}(X,Y) &=& \frac{(m-1+3\lambda^2)c+3(m-\lambda^2)+5}{4}g(U^\lambda X,U^\lambda Y)+g(T X,Y) \\
 \nonumber&+&\sum_{k=1}^{2n-m}\{(m+1)(trace A^{'}_k)g(A^{'}_kX,Y)-g(A^{'}_kX,A^{'}_kY)\}
\end{eqnarray}
for any vector fields $X,Y$ on $M$.
\end{corollary}
\begin{proof}
Putting $f_1=\frac{c+3}{4},\  f_2=f_3=\frac{c-1}{4}$ in (\ref{eqn4.2}) we obtain the result.
\end{proof}
\begin{lemma}
The scalar curvature $\tau^{'}$ of an almost semi-invariant submanifold $M$ of $\bar{M}^{2n+1}(f_1,f_2,f_3)$ with respect to $\bar{\nabla}^{'}$ is
\begin{eqnarray}
\label{eqn4.4}
\tau^{'}&=& f_1+\frac{1}{m(m+1)}\{3f_2\sum_{\lambda}^{}n(\lambda)\lambda^2-2mf_3+m\}\\
\nonumber&+&(m+1)^2||H^{'}||^2-||h^{'}||^2.
\end{eqnarray}
\end{lemma}
\begin{proof}
It is known that
\begin{equation}
\label{eqn4.7}
\tau^{'}=\frac{1}{m(m+1)}\sum\limits_{i,j=1}^{m+1}{R}^{'}(E_i,E_j,E_j,E_i).
\end{equation}
 Using (\ref{eqn2.13}), (\ref{eqn2.9b}) in (\ref{eqn4.7}) we get (\ref{eqn4.4}).
\end{proof}
\begin{theorem}
If $M$ is an almost semi-invariant minimal submanifolds of $\bar{M}^{2n+1}(f_1,f_2,f_3)$ with respect to $\bar{\nabla}^{'}$,
then the following condition holds:
\begin{itemize}
  \item [(i)]
$S^{'}(X,X)\leq\sum\limits_{\lambda}^{}(mf_1+3f_2\lambda^2-f_3+1)g(U^\lambda X,U^\lambda X)+m(f_1-f_3)\eta(X)\eta(X)+(f_1-f_3) g(T X,X)$,\\
\item [(ii)] $\tau^{'}\leq f_1+\frac{1}{m(m+1)}\{3f_2\sum\limits_{\lambda}^{}n(\lambda)\lambda^2-2mf_3+m\}$.
\end{itemize}
\end{theorem}
\begin{proof}
Since $M$ is minimal submanifold with respect to $\bar{\nabla}^{'}$, then we have
\begin{eqnarray}\label{eqn4.5}
\sum_{k=1}^{2n-m}(trace A^{'}_k)g(A^{'}_kX,X)&=&\sum_{i=1}^{m+1}g(h^{'}(X,X),h^{'}(E_i,E_i))\\
\nonumber&=&(m+1)g(h^{'}(X,X),H^{'})=0.
\end{eqnarray}
Using (\ref{eqn4.2}) and (\ref{eqn4.5}) we have
\begin{eqnarray}
\label{eqn4.6}
  &&S^{'}(X,X)-\sum_{\lambda}^{}(mf_1+3f_2\lambda^2-f_3+1)g(U^\lambda X,U^\lambda X)  \\
\nonumber&&-m(f_1-f_3)\eta(X)\eta(X)-(f_1-f_3) g(T X,X)\\
\nonumber&=&-\sum_{k=1}^{2n-m}g(A^{'}_kX,A^{'}_kX)\leq 0.
\end{eqnarray}
This proves (i).\\
The second part (ii) is comes directly from Lemma $4.3$.
\end{proof}
\begin{remark}
The equality of (i) and (ii) in Theorem $4.1$ holds if $M$ is almost semi-invariant totally geodesic submanifolds of $\bar{M}^{2n+1}(f_1,f_2,f_3)$ with respect to $\bar{\nabla}^{'}$.
 \end{remark}
\begin{proof}
If $M$ is totally geodesic submanifold with respect to $\bar{\nabla}^{'}$, then $M$ is minimal submanifold with respect to $\bar{\nabla}^{'}$. Then by virtue of Lemma $4.2$ we have the equality case of (i) and by virtue of Lemma $4.3$ we have the equality case (ii).
\end{proof}
\section{Submanifolds of $\bar{M}^{2n+1}(f_1,f_2,f_3)$ with $\hat{\bar{\nabla}}$}
\begin{lemma}
The Ricci tensor $\hat{S}$ of submanifold $M$ of $\bar{M}^{2n+1}(f_1,f_2,f_3)$ with respect to $\hat{\bar{\nabla}}$ is
\begin{eqnarray}
\label{eqn5.1}
\hat{S}(X,Y) &=& mf_1g(X,Y)+\{3f_2+(f_1-f_3)^2\}g(TX,TY)\\
\nonumber&-&\{f_3+(f_1-f_3)^2\}\{g(X,Y)+(m-1)\eta(X)\eta(Y)\} \\
\nonumber&+&\sum_{k=1}^{2n-m}\{(m+1)(trace \hat{A}_k)g(\hat{A}_kX,Y)-g(\hat{A}_kX,\hat{A}_kY)\}
\end{eqnarray}
for any vector fields $X,Y$ on $M$.
\end{lemma}
\begin{proof}
Using (\ref{eqn2.15}) and (\ref{eqn2.9c}) we have the above Lemma.
\end{proof}
\begin{lemma}
The Ricci tensor $\hat{S}$ of almost semi-invariant submanifolds $M$ of $\bar{M}^{2n+1}(f_1,f_2,f_3)$ with respect to $\hat{\bar{\nabla}}$ is
\begin{eqnarray}
\label{eqn5.2}
\hat{S}&=&\sum_{\lambda}^{}\big[mf_1+3f_2\lambda^2-f_3+(f_1-f_3)^2(\lambda^2-1)\big]g(U^\lambda X,U^\lambda Y)\\
\nonumber&+&m\big[f_1-\{f_3+(f_1-f_3)^2\}\big]\eta(X)\eta(Y)\\
\nonumber&+&\sum_{k=1}^{2n-m}\{(m+1)(trace \hat{A}_k)g(\hat{A}_kX,Y)-g(\hat{A}_kX,\hat{A}_kY)\}
\end{eqnarray}
for any vector fields $X,Y$ on $M$.
\end{lemma}
\begin{proof}
Using (\ref{eqn2.10b}) and (\ref{eqn5.1}) we have the above Lemma.
\end{proof}
\begin{corollary}
The Ricci tensor $\hat{S}$ of almost semi-invariant submanifolds $M$ of  $\bar{M}^{2n+1}(c)$ with respect to $\hat{\bar{\nabla}}$ is
\begin{eqnarray}
\label{eqn5.3}
\hat{S}(X,Y) &=& \frac{(m-1+3\lambda^2)c+3(m-1)+\lambda^2}{4}g(U^\lambda X,U^\lambda Y) \\
 \nonumber&+&2m\eta(X)\eta(Y)+\sum_{k=1}^{2n-m}\{(m+1)(trace \hat{A}_k)g(\hat{A}_kX,Y)\\
 \nonumber&-&g(\hat{A}_kX,\hat{A}_kY)\}.
\end{eqnarray}
\end{corollary}
\begin{proof}
Putting $f_1=\frac{c+3}{4},\  f_2=f_3=\frac{c-1}{4}$ in (\ref{eqn5.2}) we obtain the result.
\end{proof}
\begin{lemma}
The scalar curvature $\hat{\tau}$ of an almost semi-invariant submanifold $M$ of $\bar{M}^{2n+1}(f_1,f_2,f_3)$ with respect to $\hat{\bar{\nabla}}$ is
\begin{eqnarray}
\label{eqn5.4}
\hat{\tau} &=& f_1+\frac{1}{m(m+1)}\big[\{3f_2+(f_1-f_3)^2\}\sum_{\lambda}^{}n(\lambda)\lambda^2\\
\nonumber&-&2m\{f_3+(f_1-f_3)^2\}+(m+1)^2||\hat{H}||^2-||\hat{h}||^2\big].
\end{eqnarray}

\end{lemma}
\begin{proof}
It is known that
\begin{equation}
\label{eqn5.7}
\hat{\tau}=\frac{1}{m(m+1)}\sum\limits_{i,j=1}^{m+1}\hat{R}(E_i,E_j,E_j,E_i).
\end{equation}
Using (\ref{eqn2.11}), (\ref{eqn2.9c}) in (\ref{eqn5.7}) we get (\ref{eqn5.4}).
\end{proof}
\begin{theorem}
If $M$ is an almost semi-invariant minimal submanifolds of $\bar{M}^{2n+1}(f_1,f_2,f_3)$ with respect to $\hat{\bar{\nabla}}$, then the following condition holds:
\begin{itemize}
  \item [(i)] $\hat{S} \leq\sum\limits_{\lambda}^{}\big[mf_1+3f_2\lambda^2-f_3+(f_1-f_3)^2(\lambda^2-1)\big]g(U^\lambda X,U^\lambda Y)+m\big[f_1-\{f_3+(f_1-f_3)^2\}\big]\eta(X)\eta(Y)$,\\
  \item [(ii)] $\hat{\tau}\leq f_1+\frac{1}{m(m+1)}\big[\{3f_2+(f_1-f_3)^2\}\sum\limits_{\lambda}^{}n(\lambda)\lambda^2-2m\{f_3+(f_1-f_3)^2\}\big]$.
\end{itemize}
\end{theorem}
\begin{proof}
Since $M$ is minimal submanifold with respect to with respect to $\hat{\bar{\nabla}}$, then we have
\begin{eqnarray}\label{eqn5.5}
\sum_{k=1}^{2n-m}(trace \hat{A}_k)g(\hat{A}_kX,X)&=&\sum_{i=1}^{m+1}g(\hat{h}(X,X),\hat{h}(E_i,E_i))\\
\nonumber&=&(m+1)g(\hat{h}(X,X),\hat{H})=0.
\end{eqnarray}
Using (\ref{eqn5.2}) and (\ref{eqn5.5}) we have
\begin{eqnarray}
\label{eqn5.6}
  &&\hat{S}-\sum_{\lambda}^{}\big[mf_1+3f_2\lambda^2-f_3+(f_1-f_3)^2(\lambda^2-1)\big]g(U^\lambda X,U^\lambda Y)\\
\nonumber&&-m\big[f_1-\{f_3+(f_1-f_3)^2\}\big]\eta(X)\eta(Y)\\
\nonumber&=&-\sum_{k=1}^{2n-m}g(\hat{A}_kX,\hat{A}_kX)\leq 0.
\end{eqnarray}
This complete the proves (i).\\
The second part (ii) is comes directly from Lemma $5.3$.
\end{proof}
\begin{remark}
The equality of (i) and (ii) in Theorem $5.1$ holds if $M$ is almost semi-invariant totally geodesic submanifolds of $\bar{M}^{2n+1}(f_1,f_2,f_3)$ with respect to $\hat{\bar{\nabla}}$.
\end{remark}
\begin{proof}
If $M$ is totally geodesic submanifold with respect to $\hat{\bar{\nabla}}$, then $M$ is minimal submanifold with respect to $\hat{\bar{\nabla}}$. Then by virtue of Lemma $5.2$ we have the equality case (i) and by virtue of Lemma $5.3$ we have the equality case of (ii).
\end{proof}
\section{Submanifolds of $\bar{M}^{2n+1}(f_1,f_2,f_3)$ with $\stackrel{*}{\bar{\nabla}}$}
\begin{lemma}
The Ricci tensor $\stackrel{*}{S}$ of submanifold $M$ of $\bar{M}^{2n+1}(f_1,f_2,f_3)$ with respect to $\stackrel{*}{\bar{\nabla}}$ is
\begin{eqnarray}
\label{eqn6.1}
\stackrel{*}{S}(X,Y) &=& mf_1g(X,Y)+\{3f_2+2(f_1-f_3)+(f_1-f_3)^2\}g(TX,TY)\\
\nonumber&-&\{f_3+(f_1-f_3)^2\}\{g(X,Y)+(m-1)\eta(X)\eta(Y)\} \\
\nonumber&+&\sum_{k=1}^{2n-m}\{(m+1)(trace \stackrel{*}{A}_k)g(\stackrel{*}{A}_kX,Y)-g(\stackrel{*}{A}_kX,\stackrel{*}{A}_kY)\}
\end{eqnarray}
for any vector fields $X,Y$ on $M$.
\end{lemma}
\begin{proof}
Using (\ref{eqn2.17}) and (\ref{eqn2.9d}) we have the above Lemma.
\end{proof}
\begin{lemma}
The Ricci tensor $\stackrel{*}{S}$ of almost semi-invariant submanifolds $M$ of $\bar{M}^{2n+1}(f_1,f_2,f_3)$ with respect to $\stackrel{*}{\bar{\nabla}}$ is
\begin{eqnarray}
\label{eqn6.2}
\stackrel{*}{S}(X,Y)&=&\sum_{\lambda}^{}\big[mf_1+\{3f_2+2(f_1-f_3)\}\lambda^2-f_3\\
\nonumber&+&(f_1-f_3)^2(\lambda^2-1)\big]g(U^\lambda X,U^\lambda Y)\\
\nonumber&+&m\big[f_1-\{f_3+(f_1-f_3)^2\}\big]\eta(X)\eta(Y)\\
\nonumber&+&\sum_{k=1}^{2n-m}\{(m+1)(trace \stackrel{*}{A}_k)g(\stackrel{*}{A}_kX,Y)-g(\stackrel{*}{A}_kX,\stackrel{*}{A}_kY)\}
\end{eqnarray}
for any vector fields $X,Y$ on $M$.
\end{lemma}
\begin{proof}
Using (\ref{eqn2.10b}) and (\ref{eqn6.1}) we have the above Lemma.
\end{proof}
\begin{corollary}
The Ricci tensor $\stackrel{*}{S}$ of almost semi-invariant submanifolds $M$ of $\bar{M}^{2n+1}(c)$ with respect to $\stackrel{*}{\bar{\nabla}}$ is
\begin{eqnarray}
\label{eqn6.3}
\stackrel{*}{S}(X,Y) &=& \frac{(m-1+3\lambda^2)c+3(m-1+3\lambda^2)}{4}g(U^\lambda X,U^\lambda Y) \\
 \nonumber&+&2m\eta(X)\eta(Y)+\sum_{k=1}^{2n-m}\{(m+1)(trace \stackrel{*}{A}_k)g(\stackrel{*}{A}_kX,Y)\\
 \nonumber&-&g(\stackrel{*}{A}_kX,\stackrel{*}{A}_kY)\}
\end{eqnarray}
for any vector fields $X,Y$ on $M$.
\end{corollary}
\begin{proof}
Putting $f_1=\frac{c+3}{4},\  f_2=f_3=\frac{c-1}{4}$ in (\ref{eqn6.2}) we obtain the result.
\end{proof}
\begin{lemma}
 The scalar curvature $\stackrel{*}{\tau}$ of an almost semi-invariant submanifold $M$ of $\bar{M}^{2n+1}(f_1,f_2,f_3)$ with respect to $\stackrel{*}{\bar{\nabla}}$ is
\begin{eqnarray}
\label{eqn6.4}
\stackrel{*}{\tau} &=& f_1+\frac{1}{m(m+1)}\big[\{3f_2+2(f_1-f_3)+(f_1-f_3)^2\}\sum_{\lambda}^{}n(\lambda)\lambda^2\\
\nonumber&-&2m\{f_3+(f_1-f_3)^2\}+(m+1)^2||\stackrel{*}{H}||^2-||\stackrel{*}{h}||^2\big].
\end{eqnarray}
\begin{proof}
It is known that
\begin{equation}
\label{eqn6.7}
\stackrel{*}{\tau}=\frac{1}{m(m+1)}\sum\limits_{i,j=1}^{m+1}\stackrel{*}{R}(E_i,E_j,E_j,E_i).
\end{equation}
Using (\ref{eqn2.11}), (\ref{eqn2.9d}) in (\ref{eqn6.7}) we get (\ref{eqn6.4}).
\end{proof}
\end{lemma}

\begin{theorem}
Let $M$ be an almost semi-invariant minimal submanifolds of $\bar{M}^{2n+1}(f_1,f_2,f_3)$ with respect to $\stackrel{*}{\bar{\nabla}}$, then the following condition holds:
\begin{itemize}
  \item [(i)]
$\stackrel{*}{S}(X,Y)\leq\sum\limits_{\lambda}^{}\big[mf_1+\{3f_2+2(f_1-f_3)\}\lambda^2-f_3+(f_1-f_3)^2(\lambda^2-1)\big]g(U^\lambda X,U^\lambda X)+m\big[f_1-\{f_3+(f_1-f_3)^2\}\big]\eta(X)\eta(X)$,\\
  \item [(ii)] $\stackrel{*}{\tau}\leq f_1+\frac{1}{m(m+1)}\big[\{3f_2+2(f_1-f_3)+(f_1-f_3)^2\}\sum_{\lambda}^{}n(\lambda)\lambda^2-2m\{f_3+(f_1-f_3)^2\}\big]$.
\end{itemize}
\end{theorem}
\begin{proof}
Since $M$ is minimal submanifold with respect to with respect to $\stackrel{*}{\bar{\nabla}}$, then we have
\begin{eqnarray}\label{eqn6.5}
\sum_{k=1}^{2n-m}(trace \stackrel{*}{A}_k)g(\stackrel{*}{A}_kX,X)&=&\sum_{i=1}^{m+1}g(\stackrel{*}{h}(X,X),\stackrel{*}{h}(E_i,E_i))\\
\nonumber&=&(m+1)g(\stackrel{*}{h}(X,X),\stackrel{*}{H})=0.
\end{eqnarray}
Using (\ref{eqn6.2}) and (\ref{eqn6.5}) we have
\begin{eqnarray}
\label{eqn6.6}
 &&\stackrel{*}{S}-\sum_{\lambda}^{}\big[mf_1+\{3f_2+2(f_1-f_3)\}\lambda^2-f_3+(f_1-f_3)^2(\lambda^2-1)\big]\\
 \nonumber&&g(U^\lambda X,U^\lambda Y)-m\big[f_1-\{f_3+(f_1-f_3)^2\}\big]\eta(X)\eta(Y)\\
\nonumber&=&-\sum_{k=1}^{2n-m}g(\stackrel{*}{A}_kX,\stackrel{*}{A}_kX)\leq 0,
\end{eqnarray}
which proves (i).\\
The proof of (ii) comes directly from Lemma $6.3$.
\end{proof}
\begin{remark}
The equality of (i) and (ii) in Theorem $6.1$ holds if $M$ is almost semi-invariant totally geodesic submanifolds of $\bar{M}^{2n+1}(f_1,f_2,f_3)$ with respect to $\stackrel{*}{\bar{\nabla}}$.
\end{remark}
\begin{proof}
If $M$ is totally geodesic submanifold with respect to $\stackrel{*}{\bar{\nabla}}$, then $M$ is minimal submanifold with respect to $\stackrel{*}{\bar{\nabla}}$. Then by virtue of Lemma $6.2$ we have the equality case (i) and by virtue of Lemma $6.3$ we have the equality case of (ii).
\end{proof}
\noindent{\bf Acknowledgement:} The first author (P. Mandal)  gratefully acknowledges to
the CSIR(File No.:09/025(0221)/2017-EMR-I), Govt. of India for financial assistance.
The second author (S. K. Hui) are thankful to University of Burdwan for providing administrative and technical support.

\vspace{0.01in} \noindent Pradip Mandal and Shyamal Kumar Hui  \\
Department of Mathematics\\
The University of Burdwan\\
Burdwan -- 713104\\
West Bengal, India\\
E-mail: pm2621994@gmail.com and skhui@math.buruniv.ac.in
\end{document}